\documentclass[11pt,reqno,fleqn]{article}
\usepackage{amsfonts,amssymb,mathrsfs,amsmath,amsthm,relsize,bigints}
\usepackage[top=20mm,bottom=20mm,left=20mm,right=20mm]{geometry}
\usepackage[T1]{fontenc}
\usepackage{times}
\usepackage{multicol}
\usepackage{graphicx}
\usepackage{graphics}
\usepackage{enumitem}
\usepackage{cite}
\usepackage{fancyhdr}
\fancypagestyle{titlepage}
{
   \fancyhf{}
   \lhead{The final publication is available at "https://link.springer.com/article/10.1007/s40840-017-0577-6"}
   \rhead{}
   \fancyfoot[C]{\thepage}
}
\makeatletter
\let\@fnsymbol\@arabic
\makeatother

\newtheorem{thm}{Theorem}[section]
\numberwithin{equation}{section}

\theoremstyle{definition}

\newtheorem*{mythm}{Theorem}
\usepackage{etoolbox}
\makeatletter
\patchcmd{\@maketitle}{\begin{center}}{\begin{flushleft}}{}{}
\patchcmd{\@maketitle}{\begin{tabular}[t]{c}}{\begin{tabular}[t]{@{}l}}{}{}
\patchcmd{\@maketitle}{\end{center}}{\end{flushleft}}{}{}
\makeatother
\allowdisplaybreaks
\begin{document}
\title{Approximation of continuous periodic functions of two variables via power series methods of summability}
\author{Enes Yavuz\thanks{Department of Mathematics, Manisa Celal Bayar University, Manisa, Turkey. E-mail: enes.yavuz@cbu.edu.tr} \ \ and \"{O}zer Talo\thanks{ Muradiye Mahallesi, Yunusemre, 45140 Manisa, Turkey. E-mail: ozertalo@hotmail.com}}
\date{{\small }}
\maketitle
\thispagestyle{titlepage}

\noindent\textbf{Abstract:}
\noindent We prove a Korovkin type approximation theorem via power series methods of summability for continuous $2\pi$-periodic functions of two variables  and verify the convergence of approximating double sequences of positive linear operators by using modulus of continuity. An example concerning double Fourier series is also constructed to illustrate the obtained results.
\section{Introduction}
The classical Korovkin second theorem is stated as follows.
\begin{mythm}
Let $(L_n)$ be a sequence of positive linear operators from the space $C_{2\pi}(\mathbb{R})$ into itself. Then $\lim_n\Vert L_n(f)- f\Vert=0$ for all $f\in C_{2\pi}(\mathbb{R})$ if and only if $\lim_n\Vert L_n(f_i)- f_i\Vert=0$ for $i=0,1,2$ where  $f_0(x)=1, f_1(x)=\cos x, f_2(x)=\sin x$.
\end{mythm}
Classical versions of which are introduced by P. P. Korovkin\cite{korovkin1,korovkin2}, Korovkin type theorems deal with the approximation of functions by positive linear operators by means of providing subsets of test functions which guarantee the approximation in whole space. Following its invention, Korovkin theory has developed in many ways and found applications in various branches of mathematics. Researchers have investigated the approximation of functions in different spaces through various subsets of test functions. Besides, in connection with sequence transformations, weighted mean methods and power series methods of summability have been applied to Korovkin type theorems to recover the convergence of operators for which classical Korovkin theorems fail to work\cite{tek2,tek3,tek5,tek6,tek7,tek8}. Furthermore, there are studies dealing with Korovkin type theorems via weighted mean summability methods for continuous functions of two variables\cite{duman,demirci,cift2,cift3}. In this paper, we prove a Korovkin type approximation theorem for $2\pi$-periodic and real valued continuous functions on $\mathbb{R}^2$ by using power series methods of summability and also give an approximation theorem by using modulus of continuity. Besides we  construct an illustrative example concerning double Fourier series of functions of two variables such that our new result works but classical Korovkin theorem for periodic functions of two variables does not work.

Now we give some preliminaries concerning the space of  $2\pi$-periodic and real valued continuous functions on  $\mathbb{R}^2$  and concerning the concept of power series methods of summability for double sequences. A function $f$ on $\mathbb{R}^2$ is $2\pi$-periodic if for all $(x,y)\in \mathbb{R}^2$
\begin{eqnarray*}
f(x,y)&=&f(x+2k\pi ,y)=f(x, y+2k\pi)
\end{eqnarray*}
holds for $k=0, \pm 1,\pm 2,\ldots.$ The space of  $2\pi$-periodic and real valued continuous functions on $\mathbb{R}^2$ is denoted by $C^*(\mathbb{R}^2)$ and is equipped with the norm
\begin{eqnarray*}
\Vert f\Vert_{C^*(\mathbb{R}^2)}:=\sup_{(x,y)\in\mathbb{R}^2}|f(x,y)| \qquad (f\in C^*(\mathbb{R}^2)).
\end{eqnarray*}
Suppose that $p=(p_{mn})$ is a double sequence of nonnegative numbers with $p_{00}>0$ such that $\sum_{k,l=0}^{m,n}p_{kl}\to\infty$ as $m,n\to\infty$ and associated power series $p(r,s)=\sum_{m,n=0}^{\infty}p_{mn}r^ms^n$ is convergent for $r,s\in(0,1)$, where within the paper convergence of double sequences and of double series is meant in Pringsheim's sense. A sequence $(a_{mn})$ is said to be summable to $l$ by power series method determined by $p$ if $\sum_{m,n=0}^{\infty}a_{mn}p_{mn}r^ms^n$ converges for $r,s\in(0,1)$ and
\begin{eqnarray*}
\frac{1}{p(r,s)}\sum_{m,n=0}^{\infty}a_{mn}p_{mn}r^ms^n=l \quad as\quad r,s\to 1^-,
\end{eqnarray*}
where we write $a_{mn}\to l\ (J_p)$. Power series method $J_p$ is b-regular if for any fixed $m_0$, $n_0$
\begin{eqnarray*}
\frac{1}{p(r,s)}\sum_{m=0}^{\infty}p_{mn_0}r^m \to 0, \qquad \frac{1}{p(r,s)}\sum_{n=0}^{\infty}p_{m_0n}s^n \to 0  \quad as\quad r,s\to 1^-
\end{eqnarray*}
holds(see \cite[p. 84]{regular1}, \cite{regular2,regular3}). We note that in special cases $p_{mn}=1$ and $p_{mn}=\frac{1}{(m+1)(n+1)}$ corresponding power series methods are the Abel summability method and logarithmic summability method, respectively.

Let $\left(L_{mn}\right)$ be a double sequence of positive linear operators from $C^*(\mathbb{R}^2)$ into itself such that for every \mbox{$r,s\in(0,1)$}
\begin{eqnarray}\label{basic}
\sum_{m,n=0}^{\infty}\Vert L_{mn}(f_0)\Vert_{C^*(\mathbb{R}^2)} p_{mn}r^ms^n<\infty
\end{eqnarray}
where $f_0(x,y)=1$. Then for all $f\in C^*(\mathbb{R}^2)$ double series {\small$\sum_{m,n=0}^{\infty}L_{mn}(f;x,y)p_{mn}r^ms^n$} is convergent for $r,s\in(0,1)$.
\section{Main Results}
\begin{thm}\label{theorem1}
Let $(L_{mn})$ be a double sequence of positive linear operators from $C^*(\mathbb{R}^2)$ into  $C^*(\mathbb{R}^2)$ such that (\ref{basic}) is satisfied. Then for all $f\in C^*(\mathbb{R}^2)$,
\begin{eqnarray}\label{theorem1-1}
L_{mn}(f)\to f \ (J_p)
\end{eqnarray}
if and only if
\begin{eqnarray}\label{theorem1-2}
L_{mn}(f_i)\to f_i \ (J_p)\quad (i=0,1,2,3,4)
\end{eqnarray}
with $f_0(x,y)=1, f_1(x,y)=\sin x, f_2(x,y)=\sin y,f_3(x,y)=\cos x,f_4(x,y)=\cos y$.
\end{thm}
\begin{proof}
Let $(L_{mn})$ be a double sequence of positive linear operators from $C^*(\mathbb{R}^2)$ into  $C^*(\mathbb{R}^2)$ satisfying (\ref{basic}). First suppose that (\ref{theorem1-1}) is satisfied for all functions in $C^*(\mathbb{R}^2)$. Then in particular it is satisfied for $f_0, f_1, f_2, f_3,f_4$ since $f_i\in C^*(\mathbb{R}^2)$ and this completes the necessity part of the proof of theorem. Now we shall prove the sufficiency part. Let $f\in C^*(\mathbb{R}^2)$ and (\ref{theorem1-2}) be satisfied. Let $I$ and $J$ be closed subintervals of length $2\pi$  of $\mathbb{R}$. Fix $(x,y)\in I\times J$. It follows from the continuity of $f$ that for given $\varepsilon>0$ there is a number $\delta>0$ such that
\begin{eqnarray*}
|f(u,v)-f(x,y)|<\varepsilon+\frac{2M_f}{\sin^2\left(\frac{\delta}{2}\right)}\left\{\sin^2\left(\frac{u-x}{2}\right)+\sin^2\left(\frac{v-y}{2}\right)\right\}
\end{eqnarray*}
where $M_f=\Vert f\Vert_{C^*(\mathbb{R}^2)}$ in view of the the proof of Theorem 2.1 in \cite{duman}. Hence we obtain
\begin{eqnarray*}
\Bigg|\frac{1}{p(r,s)}\sum_{m,n=0}^{\infty}L_{mn}(f;x,y)&&\!\!\!\!\!\!\!\!p_{mn}r^ms^n-f(x,y)\Bigg|\\
&=&
\left|\frac{1}{p(r,s)}\!\!\sum_{m,n=0}^{\infty}\!\!\!L_{mn}(f;x,y)p_{mn}r^ms^n-\frac{f(x,y)}{p(r,s)}\!\!\sum_{m,n=0}^{\infty}\!\!\!L_{mn}(f_0;x,y)p_{mn}r^ms^n\right.
\\&& \
\left.+\frac{f(x,y)}{p(r,s)}\sum_{m,n=0}^{\infty}L_{mn}(f_0;x,y)p_{mn}r^ms^n-f(x,y) \right|
\\&\leq&
\frac{1}{p(r,s)}\sum_{m,n=0}^{\infty}L_{mn}\left(|f(u,v)-f(x,y)|;x,y\right)p_{mn}r^ms^n
\\&& \
+M_f\left|\frac{1}{p(r,s)}\sum_{m,n=0}^{\infty}L_{mn}(f_0;x,y)p_{mn}r^ms^n-f_0(x,y)\right|
\\&\leq&
\frac{\varepsilon}{p(r,s)}\sum_{m,n=0}^{\infty}L_{mn}(f_0;x,y)p_{mn}r^ms^n
\\&& \
+\frac{2M_f}{\sin^2\left(\frac{\delta}{2}\right)p(r,s)}\!\!\sum_{m,n=0}^{\infty}\!\!\!L_{mn}\!\left(\sin^2\left(\frac{u-x}{2}\right)+\sin^2\left(\frac{v-y}{2}\right);x,y\right)\!p_{mn}r^ms^n
\\&&
+\ M_f\left|\frac{1}{p(r,s)}\sum_{m,n=0}^{\infty}L_{mn}(f_0;x,y)p_{mn}r^ms^n-f_0(x,y)\right|
\\&\leq&
\varepsilon+(\varepsilon+M_f)\left|\frac{1}{p(r,s)}\sum_{m,n=0}^{\infty}L_{mn}(f_0;x,y)p_{mn}r^ms^n-f_0(x,y)\right|
\\&&
+\frac{M_f}{\sin^2\left(\frac{\delta}{2}\right)}\left\{2\left| \frac{1}{p(r,s)}\sum_{m,n=0}^{\infty}L_{mn}(f_0;x,y)p_{mn}r^ms^n-f_0(x,y)\right|\right.
\\&&\qquad\qquad\quad
+|\sin x|\left|\frac{1}{p(r,s)}\sum_{m,n=0}^{\infty}L_{mn}(f_1;x,y)p_{mn}r^ms^n-f_1(x,y)\right|
\\&& \qquad\qquad\quad
 +|\sin y|\left|\frac{1}{p(r,s)}\sum_{m,n=0}^{\infty}L_{mn}(f_2;x,y)p_{mn}r^ms^n-f_2(x,y)\right|
 \\&&\qquad\qquad\quad
+|\cos x|\left|\frac{1}{p(r,s)}\sum_{m,n=0}^{\infty}L_{mn}(f_3;x,y)p_{mn}r^ms^n-f_3(x,y)\right|
\\&&\qquad\qquad\quad
 \left. +|\cos y|\left|\frac{1}{p(r,s)}\sum_{m,n=0}^{\infty}L_{mn}(f_4;x,y)p_{mn}r^ms^n-f_4(x,y)\right|\right\}
\\&\leq&
 \varepsilon+K\mathlarger{\sum_{i=0}^4} \left|\frac{1}{p(r,s)}\sum_{m,n=0}^{\infty}L_{mn}(f_i;x,y)p_{mn}r^ms^n-f_i(x,y)\right|
\end{eqnarray*}%
where $K=\varepsilon+ M_f+ 2M_f/\sin^2(\delta/2)$ by using the fact that $\sin^2\left(\frac{a-b}{2}\right)=\frac{1}{2}(1-\cos a\cos b-\sin a\sin b)$. Then taking supremum over $(x,y)$ we conclude
\begin{eqnarray*}
\Bigg\Vert\frac{1}{p(r,s)}\sum_{m,n=0}^{\infty}L_{mn}(f)p_{mn}r^ms^n-f\Bigg\Vert_{C^*(\mathbb{R}^2)}
\leq \varepsilon + K\mathlarger{\sum_{i=0}^4} \left\Vert\frac{1}{p(r,s)}\sum_{m,n=0}^{\infty}L_{mn}(f_i)p_{mn}r^ms^n-f_i\right\Vert_{C^*(\mathbb{R}^2)}
\end{eqnarray*}
and this completes the proof.
\end{proof}
Now we give a theorem concerning the convergence of the sequence of positive linear operators acting on the space $C^*(\mathbb{R}^2)$ with the help of modulus of continuity. For $f\in C^*(\mathbb{R}^2)$ and for $(u,v),(x,y)\in\mathbb{R}^2$ the modulus of continuity of $f$ is defined by
\begin{eqnarray*}
\omega(f; \delta):=\sup\left\{|f(u,v)-f(x,y)|: \sqrt{(u-x)^2+(v-y)^2}\leq \delta\right\}
\end{eqnarray*}
for any $\delta>0$.
\begin{thm}\label{theorem2}
Let $(L_{mn})$ be a double sequence of positive linear operators from $C^*(\mathbb{R}^2)$ into  $C^*(\mathbb{R}^2)$ such that (\ref{basic}) is satisfied. If
\begin{itemize}
\item[(i)] $L_{mn}(f_0)\to f_0 \ (J_p)$
\item[(ii)] $\lim\limits_{r,s\to1^-}\omega(f; \gamma(r,s))=0$ where  {\small$\gamma(r,s)=\sqrt{\left\Vert\frac{1}{p(r,s)}\sum_{m,n=0}^{\infty}L_{mn}(\varphi)p_{mn}r^ms^n\right\Vert_{C^*(\mathbb{R}^2)}}$}
\end{itemize}
with $\varphi(u,v)=\sin^2\left(\frac{u-x}{2}\right)+\sin^2\left(\frac{v-y}{2}\right)$, then for all $f\in C^*(\mathbb{R}^2)$
\begin{eqnarray*}
L_{mn}(f)\to f \ (J_p).
\end{eqnarray*}
\end{thm}
\begin{proof}
Let $f\in C^*(\mathbb{R}^2)$  and $(x,y)\in[-\pi, \pi]\times[-\pi, \pi]$. Then in view of the proof of Theorem 9 in \cite{demirci} we have
\begin{eqnarray*}
|f(u,v)-f(x,y)|\leq \left(1+\pi^2\frac{\sin^2\left(\frac{u-x}{2}\right)+\sin^2\left(\frac{v-y}{2}\right)}{\delta^2}\right)\omega(f; \delta)
\end{eqnarray*}
and followingly we get
\begin{eqnarray*}
\Bigg|\frac{1}{p(r,s)}\sum_{m,n=0}^{\infty}L_{mn}(f;x,y)&&\!\!\!\!\!\!\!\!p_{mn}r^ms^n-f(x,y)\Bigg|\\
\\&\leq&\frac{1}{p(r,s)}\sum_{m,n=0}^{\infty}L_{mn}\left(|f(u,v)-f(x,y)|;x,y\right)p_{mn}r^ms^n
\\&& \
+M_f\left|\frac{1}{p(r,s)}\sum_{m,n=0}^{\infty}L_{mn}(f_0;x,y)p_{mn}r^ms^n-f_0(x,y)\right|
\\&\leq&
\frac{\omega(f; \delta)}{p(r,s)}\!\!\sum_{m,n=0}^{\infty}\!\!\!L_{mn}\!\left(1+(\pi/\delta)^2\left\{\sin^2\left(\frac{u-x}{2}\right)\!+\!\sin^2\left(\frac{v-y}{2}\right)\right\};x,y\right)\!p_{mn}r^ms^n
\\&& \
+M_f\left|\frac{1}{p(r,s)}\sum_{m,n=0}^{\infty}L_{mn}(f_0;x,y)p_{mn}r^ms^n-f_0(x,y)\right|
\\&=&
\left|\frac{1}{p(r,s)}\sum_{m,n=0}^{\infty}L_{mn}(f_0;x,y)p_{mn}r^ms^n-f_0(x,y)\right|\omega(f; \delta) +\omega(f; \delta)
\\&& \
+\frac{\pi^2}{\delta^2}\frac{\omega(f; \delta)}{p(r,s)}\sum_{m,n=0}^{\infty}L_{m,n}\left(\varphi;x,y\right)p_{mn}r^ms^n
\\&& \
+M_f\left|\frac{1}{p(r,s)}\sum_{m,n=0}^{\infty}L_{mn}(f_0;x,y)p_{mn}r^ms^n-f_0(x,y)\right|
\end{eqnarray*}
as in the proof of Theorem \ref{theorem1}. Taking supremum over $(x,y)$ and putting $\delta:=\gamma(r,s)$ we conclude that
\begin{eqnarray*}
\Bigg\Vert\frac{1}{p(r,s)}\sum_{m,n=0}^{\infty}&&\!\!\!\!\!\!L_{mn}(f)p_{mn}r^ms^n-f\Bigg\Vert_{C^*(\mathbb{R}^2)}
\\&\leq&
\left\Vert\frac{1}{p(r,s)}\!\!\sum_{m,n=0}^{\infty}\!\!\!L_{mn}(f_0)p_{mn}r^ms^n-f_0\right\Vert_{C^*(\mathbb{R}^2)}\!\!\!\!\omega(f; \gamma(r,s)) + (1+\pi^2)\omega(f; \gamma(r,s))
 \\&& \quad +
M_f\left\Vert\frac{1}{p(r,s)}\sum_{m,n=0}^{\infty}L_{m,n}(f_0)p_{mn}r^ms^n-f_0\right\Vert_{C^*(\mathbb{R}^2)}
 \\&\leq&
 K\left\{\left\Vert\frac{1}{p(r,s)}\sum_{m,n=0}^{\infty}L_{mn}(f_0)p_{mn}r^ms^n-f_0\right\Vert_{C^*(\mathbb{R}^2)}\omega(f; \gamma(r,s)) + \omega(f; \gamma(r,s))\right.
 \\&& \qquad \quad
 \left.+\left\Vert\frac{1}{p(r,s)}\sum_{m,n=0}^{\infty}L_{mn}(f_0)p_{mn}r^ms^n-f_0\right\Vert_{C^*(\mathbb{R}^2)}\right\}
\end{eqnarray*}
where $K=\max\{1+\pi^2, M_f\}$. Finally by taking limit as $r,s\to 1^-$ the proof is completed by the assumptions (i) and (ii) of the theorem.
\end{proof}

\section{Illustrative example}
Let $f\in C^*(\mathbb{R}^2)$ and $\left(S_{mn}(f)\right)$  be the sequence of $mn-$th partial sums of double Fourier series
\begin{eqnarray*}
\frac{a_{00}}{4}&+&\frac{1}{2}\sum_{j=1}^{\infty}(a_{j0}\cos jx+c_{j0}\sin jx) +\frac{1}{2}\sum_{k=1}^{\infty}(a_{0k}\cos ky+b_{0k}\sin ky)\\ &+&\mathlarger{\sum_{j=1}^{\infty}\sum_{k=1}^{\infty}}\left\{
   \begin{array}{ll}
     a_{jk}\cos jx\cos ky+b_{jk}\cos jx\sin ky\\
     +c_{jk}\sin jx\cos ky+d_{jk}\sin jx\sin ky
   \end{array}
   \right\}
\end{eqnarray*}
of $f$ where $a_{jk}, b_{jk}, c_{jk}, d_{jk}$'s are the real double Fourier coefficients. We know that sequence $\left(S_{mn}(f)\right)$ does not have to converge to  $f$ neither uniformly nor pointwise in general and many studies have been done concerning the conditions ensuring the convergence. Besides some averaging processes have been applied to recover the convergence of double Fourier series. Now we apply Theorem \ref{theorem1} with $p_{mn}= 1$ to the operator
\begin{eqnarray}\label{example}
L_{mn}(f; x,y)=(1+(-1)^{m+n})T_{mn}(f;x,y)
\end{eqnarray}
where $(T_{mn})$ is the {\small $(\frac{m}{m+1}, \frac{n}{n+1})-$}th Abel-Poisson mean of double Fourier series of $f$. Then, it follows that
\begin{eqnarray*}
L_{mn}(f;x,y)\!=\!
\left(1+(-1)^{m+n}\right)\left\{\frac{1}{4\pi^2}\huge\bigintss_{-\pi}^{\pi}\!\!\!\huge\bigintss_{-\pi}^{\pi}\!\!\!\!f(x-u,y-v)P\!\left(\frac{m}{m+1},u\right)\!P\left(\frac{n}{n+1},v\right)dudv\right\}
\end{eqnarray*}
by the help of Abel-Poisson kernel $P(r,t)=\frac{1-r^2}{(1-2r\cos t+r^2)}\cdot$ Condition \eqref{basic} is satisfied since  we have $L_{mn}(f_0;x,y)=1+(-1)^{m+n}$. Besides, since
\begin{eqnarray*}
&&\!\!L_{mn}(f_1;x,y)\!=\!\left(1+(-1)^{m+n}\right)\!\left(\frac{m}{m+1}\right)\sin x,\quad L_{mn}(f_2;x,y)\!=\!\left(1+(-1)^{m+n}\right)\!\left(\frac{n}{n+1}\right)\sin y,\\
&&\!\!L_{mn}(f_3;x,y)\!=\!\left(1+(-1)^{m+n}\right)\!\left(\frac{m}{m+1}\right)\cos x,\quad L_{mn}(f_4;x,y)\!=\!\left(1+(-1)^{m+n}\right)\!\left(\frac{n}{n+1}\right)\cos y,
\end{eqnarray*}
we get
\begin{eqnarray*}
\left\Vert (1-r)(1-s)\sum_{m,n=0}^{\infty}L_{mn}(f_i)r^ms^n-f_i\right\Vert_{C^*(\mathbb{R}^2)}=
\begin{cases}
\frac{(1-r)(1-s)}{(1+r)(1+s)}, \quad &if \quad i=0\\
E(r,s),  &if \quad i=1,3\\
E(s,r) &if\quad i=2,4
\end{cases}
\end{eqnarray*}
where $E(x,t)=\left| \frac{(1-x)(1-t)}{(1+x)(1+t)}+\frac{(1-x)\ln(1-x)}{x}-\frac{(1-t)(1-x)\ln(1+x)}{(1+t)x}\right|$. Then we obtain
\begin{eqnarray*}
\left\Vert (1-r)(1-s)\sum_{m,n=0}^{\infty}L_{mn}(f_i)r^ms^n-f_i\right\Vert_{C^*(\mathbb{R}^2)}\to 0 \quad as\quad r,s\rightarrow 1^-
\end{eqnarray*}
for $i=0,1,2,3,4$ and by Theorem \ref{theorem1} we conclude that $L_{mn}(f)\to f\ (J_1)$. However, classical Korovkin theorem for periodic functions of two variables fails to work for $L_{mn}$ since even $L_{mn}(f_0)=(1+(-1)^{m+n})\nrightarrow f_0$ as $m,n\to\infty$.

Now we verify the approximation of sequence of linear operators $(L_{mn})$ in (\ref{example}) by the help of Theorem \ref{theorem2}. $L_{mn}(f_0)\to f_0 \ (J_1)$ is satisfied since
\begin{eqnarray*}
\left\Vert (1-r)(1-s)\sum_{m,n=0}^{\infty}L_{mn}(f_0)r^ms^n-f_0\right\Vert_{C^*(\mathbb{R}^2)}=\frac{(1-r)(1-s)}{(1+r)(1+s)}
\end{eqnarray*}
and so (i) of Theorem \ref{theorem2} holds. Now consider (ii) of Theorem \ref{theorem2}. Since we have
\begin{eqnarray*}
\gamma(r,s)&=&\sqrt{\left\Vert(1-r)(1-s)\!\!\sum_{m,n=0}^{\infty}\!\!\!L_{mn}\left(\sin^2\left(\frac{u-x}{2}\right)+\sin^2\left(\frac{v-y}{2}\right)\right)r^ms^n\right\Vert_{C^*(\mathbb{R}^2)}}
\\&=&
\sqrt{\left\Vert(1-r)(1-s)\!\!\sum_{m,n=0}^{\infty}\!\!\!\left\{(1+(-1)^{m+n})\!\!\left(\frac{1}{2(m+1)}\!+\!\frac{1}{2(n+1)}\right)\right\}\!r^ms^n\right\Vert_{C^*(\mathbb{R}^2)}}
\end{eqnarray*}
and since the power series method $J_p$ is regular we get that $\lim\limits_{r,s\to1^-} \gamma(r,s)=0$. Then from uniform continuity of $f$ we conclude that  $\lim\limits_{r,s\to1^-}\omega(f; \gamma(r,s))=0$ and (ii) of Theorem \ref{theorem2} is satisfied. So $L_{mn}(f)\to f \ (J_1)$.


\begin{thebibliography}{99}
\bibitem{cift3} Alotaibi A., Mursaleen M., Mohiuddine S.A.: Statistical Approximation for Periodic Functions of Two Variables. Journal of Function Spaces and Applications {\bf 2013}, 1--5 (2013)
\bibitem{regular1}Baron S., Tietz H.:  Produkts$\ddot{a}$tze f\"{u}r Verfahren zur Limitierung von Doppelfolgen. Anal. Math. {\bf 20}, 81--94 (1994)
\bibitem{regular2}Baron S., Stadtm\"{u}ller U.: Tauberian Theorems for Power Series Methods Applied to Double Sequences. J. Math. Anal. Appl. {\bf 211(2)}, 574--589 (1997)
\bibitem{cift2} Belen C., Mursaleen M., Y{\i}ld{\i}r{\i}m M.: Statistical $A-$summability of double sequences and a Korovkin type approximation theorem. Bull. Korean Math. Soc. {\bf 49(4)}, 851–861 (2012)
\bibitem{tek3} Braha N.L., Srivastava H.M., Mohiuddine S.A.:  A Korovkin's type approximation theorem for periodic functions via the statistical summability of the generalized de la Vallée Poussin mean. Appl. Math. Comput. {\bf 228}, 162--169 (2014)
\bibitem{demirci} Demirci K., Dirik F.: Four-dimensional matrix transformation and rate of A-statistical convergence of periodic functions. Math. Comput. Modelling {\bf 52}, 1858--1866 (2010)
\bibitem{duman} Duman O.,  Erku\c{s} E.: Approximation of continuous periodic functions via statistical convergence. Comput. Math. Appl. {\bf 52}, 967--974 (2006)
\bibitem{tek8} Kadak U.: On relative weighted summability in modular function spaces and associated approximation theorems. Positivity (2017). doi: 10.1007/s11117-017-0487-8
\bibitem{korovkin1} Korovkin P.P.: Convergence of linear positive operators in the spaces of continuous functions. Dokl. Akad. Nauk. SSSR (N.S.) {\bf 90}, 961--964 (1953)
\bibitem{korovkin2} Korovkin P.P.: Linear Operators and Approximation Theory. Hindustan Publ. Co., Delhi, (1960)
\bibitem{tek2} Mohiuddine S.A., Alotaibi A., Mursaleen, M.: Statistical summability $(C, 1)$ and a Korovkin type approximation theorem. J. Inequal. Appl. {\bf 172}, 1--8 (2012)
\bibitem{tek6} \"{O}zg\"{u}\c{c} \.{I}., Ta\c{s} E.: A Korovkin-Type Approximation Theorem and Power Series Method. Results. Math. {\bf 69}, 497--504 (2016)
\bibitem{regular3}Stadtm\"{u}ller U., Tali A.: A family of generalized Norlund methods and related power series methods applied to double sequences. Math. Nachr. {\bf 282(2)}, 288--306 (2009)
\bibitem{tek7} Ta\c{s} E., Yurdakadim T.: Approximation by positive linear operators in modular spaces by power series method. Positivity (2017) doi: 10.1007/s11117-017-0467-z
\bibitem{tek5} Yurdakadim T.: Some Korovkin type results via power series method in modular spaces. Commun. Fac. Sci. Univ. Ank. Ser. A1 Math. Stat. {\bf 65(2)}, 65--76 (2016)
\end{thebibliography}
\end{document}